\documentclass[12pt]{amsart}

\makeatletter
\@namedef{subjclassname@2010}{%
  \textup{2010} Mathematics Subject Classification}
\makeatother

\usepackage{amsmath}
\usepackage{amsfonts,amssymb,amsthm,amscd,latexsym,euscript}

\usepackage[all]{xy}
\SelectTips{cm}{10}

\newcommand{\Sp}{\ensuremath{\textup{Sp}}}
\newcommand{\proSp}{\ensuremath{\textup{pro-Sp}}}

\newcommand{\op}{{\ensuremath{\textup{op}}}}

\DeclareMathOperator{\Hor}{\ensuremath{\textup{Hor}}}

\newcommand{\dgrm}[1]{\ensuremath{\smash{\underset{\widetilde{\hphantom{#1}}}{#1}} \mathstrut}}

\newcommand {\cofib} {\ensuremath{\hookrightarrow}}

\newcommand {\fibr} {\ensuremath{\twoheadrightarrow}}

\newcommand {\trivcofib} {\ensuremath{\tilde\hookrightarrow}}

\newcommand {\trivfibr} {\ensuremath{\tilde\twoheadrightarrow}}

\newcommand {\we} {\ensuremath{\tilde\rightarrow}}

\DeclareMathOperator{\hocolim}{\textup{hocolim}}

\newcommand{\cal}[1]{\ensuremath{\mathcal #1}} 

\newtheorem {theorem1}{Theorem}[section]
\newtheorem {theorem}[theorem1]{Theorem}

\newtheorem {corollary}[theorem1]{Corollary}
\newtheorem {proposition}[theorem1]{Proposition}
\newtheorem {lemma}[theorem1]{Lemma}
\theoremstyle{definition}
\newtheorem {definition}[theorem1]{Definition}

\theoremstyle{remark}
\newtheorem {remark}[theorem1]{Remark}

\newcommand{\cat}[1]{\ensuremath{\EuScript #1}}

\newcommand{\colim}{\ensuremath{\mathop{\textup{colim}}}}

\renewcommand{\hom}{\ensuremath{{\rm hom}}}

\newcounter{zahl}%
    {\end{list}}%

\begin{document}

\baselineskip=17pt



\title[A classification of small homotopy functors]{A classification of small homotopy functors from spectra to spectra}

\author[B.~Chorny]{Boris Chorny}

\address{Department of Mathematics, Physics and Computer Science\\ University of Haifa at Oranim\\
Tivon 36006, Israel}

\email{chorny@math.haifa.ac.il}

\date{}

\begin{abstract}
We show that every small homotopy functor from spectra to spectra is weakly equivalent to a filtered colimit of
representable functors represented in cofibrant spectra. Moreover, we present this classification as a
Quillen equivalence of the category of small functors from spectra to spectra equipped with the homotopy
model structure and the opposite of the pro-category of spectra with the strict model structure.
\end{abstract}

\subjclass[2010]{Primary 18G55; Secondary 55P42}

\keywords{model categories, homotopy functors, pro-spectra}

\maketitle

\section{Introduction}
Let $\Sp$ denote the closed symmetric monoidal model category of spectra, which is also combinatorial.
Either symmetric spectra, \cite{Hovey-Shipley-Smith}, or Lydakis' linear functors from finite pointed simplicial sets to
simplicial sets, \cite{Lydakis}, may serve as a model.

In this work, we suggest a classification of small homotopy
functors from spectra to spectra. Namely, we show that, up to a
weak equivalence, every small homotopy functor is a filtered
colimit of representable functors represented in cofibrant
spectra.

Our interest in this question stems from classification problems
related to Goodwillie's calculus of homotopy functors. Finitary
linear (more generally, homogeneous) functors from spaces to
spaces or spectra were classified by T.~Goodwillie,
\cite{Goo:calc3}. Finitary polynomial functors were classified by
W.G.~Dwyer and C.~Rezk (unpublished) and, independently, by
G.~Arone and M.~Ching, \cite{Arone-Ching-classification}. Small
functors are rather like finitary functors, except that they
commute with filtered colimits starting from a certain non-fixed
cardinality instead of commuting with all filtered colimits as
finitary functors do. It is a natural question whether these
classifications extend to more general small functors.

In this work, we present a classification of small linear functors
from spectra to spectra. Since homotopy pushouts are also homotopy
pullbacks in $\Sp$, every representable functor is linear (if it
is represented by a cofibrant spectrum and we look at its values
only in fibrant spectra), and so are filtered colimits of
representable functors. The purpose of this work is to show that
these are all small linear functors.

It turns out that a small homotopy functor is linear, since small
functors are continuous with respect to the spectral enrichment.
This fact is a topological counterpart of a well-known algebraic
phenomenon: any additive functor preserving quasi-isomorphisms of
chain complexes gives rise to a triangulated functor of derived
categories, the total derived functor. Even though its proof seems
to be missing from the literature, it is well-known to the
experts. We are grateful to Michael Ching who brought this fact to
our attention.

Our result may be viewed as a higher version of a well known
statement about homology functors defined on the homotopy category
of spectra: every homology functor is a filtered colimit of
representables, \cite[4.19]{Hovey-Strickland}. From this point of
view, the current work continues to transfer the representability
theorems into the enriched realm, which was initiated in
\cite{Duality}, \cite{Chorny-Brownrep},
\cite{Jardine-representability}.

Of course, the most convenient way to formulate our classification
result is to exhibit it as a Quillen equivalence of certain model
categories. Indeed, we define a new model structure on the category of small functors from spectra to spectra, localizing the fibrant-projective model structure, \cite{Duality}.  In the new model structure the fibrant objects are the homotopy functors, therefore we call it the homotopy model structure. Next, we construct a Quillen pair
\[
\xymatrix{
O\colon \Sp^{\Sp} \ar@/^/[r] &  (\proSp)^{\op}: \! P, \ar@/^/[l]
}
\]
where the right adjoint $P$ is the restriction of Yoneda
embedding, sending every pro-space into a filtered colimit of
representable functors. The classification of homotopy functors
may be performed without using much of the model categories
technique, and therefore we postpone the proof that this Quillen
adjunction is a Quillen equivalence to the end of the paper.

It is interesting to compare our classification with Goodwillie's
classification of finitary linear functors, according to which,
every finitary linear functor is equivalent to $-\wedge E$ for
some spectrum $E$, so that the homotopy category of finitary
linear functors is equivalent to the homotopy category of spectra.
See \cite{BCR} for the model categorical formulation of this
classification. However, every spectrum is a filtered colimit of
compact spectra, say $E=\colim_i E_i$. Hence, $-\wedge E =
\colim -\wedge E_i$ with $E_i$ compact for every $i$. A version of
Spanier-Whitehead duality, \cite[7.1]{Duality}, ensures that $-
\wedge E_i = R^{DE_i}(-)$, and hence $-\wedge E = \colim_i
R^{DE_i}$, which fits our description. The embedding of finitary
functors into all small functors corresponds to the embedding of
spectra into the opposite category of pro-spectra as filtered
colimits of compact spectra, whose category is self-dual.

However, not every linear functor is small. For example, consider the functor $F(-)=\hom(\hom(-,A),B)$ for two fibrant spectra $A$ and $B$. It is not equivalent to a filtered colimit of representable functors, since it is not small (does not commute with filtered colimits of any size).

Another example, suggested by the anonymous referee, of a non-small functor is following: $F(Y)= \tilde Y\wedge \tilde Y$. The reason in this case is different: this functor is not continuous with respect to the spectral enrichment, i.e. not a \emph{spectral} functor. Suppose for contradiction that $F$ is a small spectral functor. Since $F$ is a homotopy functor, it may be approximated by a $C$-cellular functor $\tilde F \we F$ in the fibrant-projective model structure, \cite[5.8]{Duality}, where 
\[
C=\left\{
R^X\wedge K\cofib R^X\wedge L \left|
	\begin{array}{l}
		X \text{-- fibant and cofibrant spectrum;}\\
		K\cofib L \text{ generating cofibration in \Sp.}
	\end{array}
\right.
\right\}
\] 

Enriched version of the Spanier-Whitehead duality, \cite[Lemma~7.2]{Duality}, implies, by induction on the skeleton of $\tilde F$, that for any compact cofibrant spectrum $A$, $F(A\wedge Y)\simeq A\wedge F(Y)$. On the other hand, $F(A\wedge Y) = (A\wedge Y)_{\text{cof}} \wedge (A\wedge Y)_{\text{cof}} \simeq (\tilde A \wedge \tilde A) \wedge (\tilde Y\wedge \tilde Y)$, leading to a contradiction.

This example demonstrates  a failed attempt to define a quadratic functor, which is not linear. Indeed, as we mentioned before, we will show in Proposition~\ref{homotopy-linear} that all small homotopy functors are linear, so there is no calculus of functors for the spectral functors, just linear algebra.

The situation is entirely different for simplicial functors from spectra to spectra, \cite{Ching}. The classification of simplicial linear functors demands the development of additional technical tools and will appear in a separate paper, \cite{Chorny-ClassLinFun}.

\subsection*{The paper is organized as follows} Section~\ref{preliminaries} is devoted to
the construction of a left adjoint to $P$, which embeds the
opposite of the category of pro-spectra as a full subcategory of
pro-representable functors in $\Sp^\Sp$. Beginning with the
fibrant-projective model structure, \cite{Duality}, on the
category of small functors, we show that this adjunction is a
Quillen pair if pro-spectra are equipped with the strict model
structure, \cite{Isaksen-strict}. 

In Section~\ref{homotopy-model}, we localize the fibrant-projective model structure on the category
of small functors with respect to a proper class of maps, ensuring
that the local objects are precisely the fibrant homotopy
functors. Therefore, we name it the homotopy model structure.

Sections~\ref{section-homotopy-linear} and
\ref{section-classification} are the technical heart of the paper,
where the classification of small homotopy and linear functors is
performed, except for the model categorical reformulation.
Section~\ref{section-homotopy-linear} contains the proof that
every homotopy functor is linear, filling the gap in the
literature. Section~\ref{section-homotopy-linear} is devoted to
the proof that every linear functor is weakly equivalent to a
filtered colimit of representables in the fibrant-projective model
structure. These sections rely on a minimal model categorical
technique and hopefully may be read by people not interested in
model categories. 

It is not immediate to show that the constructed Quillen pair is a Quillen equivalence again. In order to do so, we give an alternative localization construction in Section~\ref{section-localization-construction}, which is
expressed in terms of the adjoint functors with which we are
working, and which may be described as a derived unit of this
adjunction. Using our classification, we show that this adjunction
coincides, up to homotopy, with the adjunction constructed in
Section~\ref{homotopy-model}.

Finally, in Section~\ref{section-Quillen-equivalence}, we prove our main
result, that our homotopy model structure on the category of small
functors is Quillen-equivalent to the opposite of the category of
pro-spectra. It is also formulated as Theorem~\ref{main-theorem}.

We would like to conclude this introduction with a notice that, unlike in \cite{PhDI,PhDII},
the localization with respect to a proper class of maps appearing
in this paper is not functorial. We do not know whether it is
possible to find a localization functor inverting the same class
of maps, but we have developed, with Georg Biedermann, an extension
to Bousfield-Friedlander localization machinery, which is suitable
for work with some non-functorial localization constructions,
\cite[Appendix A]{Duality}. We apply this machinery to the
localization construction from Section~\ref{homotopy-model},
whereas we were not able to apply it to the localization construction from Section~\ref{section-localization-construction}.

\subsection*{Acknowledgments} We thank Michael Ching and Bill Dwyer for helpful discussions.

\section{Preliminaries on pro-spaces}\label{preliminaries}
The goal of this preliminary section is to show that the opposite
of the category of pro-spectra is equivalent to a reflective
subcategory of small functors from spectra to spectra. If we
choose to work with the fibrant-projective model structure on the
category of small functors, this adjunction carries over to the
level of homotopy categories.

The objects of the category of pro-spectra are cofiltered diagrams
of spectra, i.e., for every filtering $I$, any functor $X\colon
I^{\op}\to \Sp$ is a pro-spectrum. We denote this pro-object as
$X_{\bullet} = \{X_i\}_{i\in I}$.

The morphisms between two pro-objects $\{X_i\}_{i\in I}$ and
$\{Y_j\}_{j\in J}$ are ladders that commute after a composition
with the bonding maps. Formally,
\[
\hom_{\proSp}(\{X_i\}, \{Y_j\}) = \lim_{j\in J}\colim_{i\in I}\hom_{\Sp}(X_i, Y_j).
\]
The category of pro-spectra is enriched over the category of spectra with the internal $\hom_{\proSp}( -, -)$
calculated by the above rule, while taking $\hom_{\Sp}( -, -)$ to be the internal $\hom$-functor in the close
symmetric monoidal category of spectra.

The category of small functors from spectra to spectra $\Sp^{\Sp}$ consists of small functors
as objects and natural transformations as morphisms. We remind the reader that a functor $F\colon \Sp\to \Sp$ is small
if it is a left Kan extension of its restriction to some small subcategory; equivalently, small functors are small
weighted colimits of representable functors.

The restriction of the Yoneda embedding on the category of spectra
is a functor $P\colon (\proSp)^{\op} \to \Sp^{\Sp}$ that sends
every pro-spectrum $X_{\bullet}$ into the pro-representable
functor $\hom_{\proSp}(X_{\bullet}, - )\colon \Sp \to \Sp$. By the
definition of morphisms in the category of pro-spectra, the
pro-representable functor $\hom_{\proSp}(\{X_i\}, - )=\colim_{i\in
I}\hom_{\Sp}(X_i,-)$ is a filtered colimit of representable
functors $R^{X_{i}}$ over $I$. In particular, every
pro-representable functor is small.

Now, we show that the functor $Y$ has a left adjoint. The argument
we give below works for every locally presentable, closed symmetric monoidal category and not just spectra.

\begin{proposition}\label{adjunction}
The functor $P\colon (\proSp)^{\op} \to \Sp^{\Sp}$ has a left adjoint $O\colon \Sp^{\Sp}\to (\proSp)^{\op}$.
\end{proposition}
\begin{proof}
We shall use the adjunction constructed in \cite{Duality}
\[\xymatrix{
Y\colon \Sp^{\op} \ar@/^/[r] & \Sp^{\Sp}: \! Z, \ar@/^/[l]
}\]
and the fact that the category of small functors from spectra to spectra is class-finitely
presentable \cite{Chorny-Rosicky}.

Every small functor is a filtered colimit of finite weighted
colimits of representable functors. Let $\Sp^{\Sp}\ni F =
\colim_{i\in I} C_i$, where $C_i=A_{k}\star_{k\in K} R^{B_k}$ with
all $A_k$ finite spectra. Then,
\begin{multline*}
\hom_{\Sp^{\Sp}}(F, PX_{\bullet}) = \hom_{\Sp^{\Sp}}(\colim_{i\in I} C_i, \colim_{j\in J}R^{X_j}) = \lim_{i\in I}\hom_{\Sp^{\Sp}}(C_i, \colim_{j\in J}R^{X_j}) = \\
\lim_{i\in I}\colim_{j\in J}\hom_{\Sp^{\Sp}}(C_i, R^{X_j}) = \lim_{i\in I}\colim_{j\in J}\hom_{\Sp^{\op}}(ZC_i, X_j) = \\
\lim_{i\in I}\colim_{j\in J}\hom_{\Sp}(X_j, ZC_i) = \hom_{\proSp}(\{X_j\},\{ZC_i\}) = \hom_{(\proSp)^{\op}}(\{ZC_i\}, \{X_j\}).
\end{multline*}

Of course, the representation of $F$ as a filtered colimit of compact objects is not unique, but if we can take any representation of the kind $F=\colim_{i\in I}C_i$, then the map $f\colon
F\to\colim_{i\in I}R^{ZC_i} = P\{ZC_i\}$ serves as a solution set, since, according to the computation above, every map $F\to PX_{\bullet}$ factors through $f$. Freyd's adjoint functor theorem implies the existence of the left adjoint for $P$, and we can compute its value, up to an isomorphism, by choosing a representation for $F$ as a filtered colimit of compact objects and declaring $OF = \{ZC_i\}_{i\in I}$.
\end{proof}

The category of small functors from spectra to spectra carries the
fibrant-projective model structure constructed in \cite{Duality}.
Fibrant-projective weak equivalences and fibrations are the
natural transformations of functors inducing levelwise weak
equivalences or fibrations between their values in fibrant
objects. We conclude the categorical preliminaries by the
following proposition that states, essentially, that the opposite
of the homotopy category of pro-spectra is a co-reflective
subcategory of the homotopy category of small functors.

\begin{proposition}\label{Quillen-map}
The pair of adjoint functors
\[
\xymatrix{
P\colon (\proSp)^{\op} \ar@/^/[r] & \Sp^{\Sp} : \! O, \ar@/^/[l]
}
\]
constructed in Proposition~\ref{adjunction} is a Quillen pair if we equip the category of small functors
with the fibrant-projective model structure and the category pro-spectra with the strict model structure.
\end{proposition}
\begin{proof}
It suffices to show that the right adjoint $P$ preserves fibrations and trivial fibrations of pro-spectra.

Consider a trivial fibration or a fibration $f^{\op}\colon
Y_{\bullet}\to X_{\bullet}$ in the opposite category of the
pro-spectra, i.e., $f\colon X_{\bullet}\to Y_{\bullet}$ is a
trivial cofibration or a cofibration in the strict model
structure, which means $f$ is an essentially levelwise trivial
cofibration or an essentially levelwise cofibration, where the
word `essentially' means `up to reindexing'.

Let $f_{i}\colon X_{i}\to Y_{i}$, $i\in I$ be a levelwise trivial
cofibration or a levelwise cofibration representing $f$. Recall
that $PX_{\bullet}=\colim_{i\in I}R^{X_i}$,
$PY_{\bullet}=\colim_{i\in I}R^{Y_i}$. Then, $Pf\colon
\colim_{i\in I}R^{X_i}\to \colim_{i\in I}R^{Y_i}$ is a trivial
fibration or a fibration, respectively, in the fibrant-projective
model structure, since each $f_{i}$ induces a trivial fibration or
a fibration of representable functors in the fibrant-projective
model structure, and filtered colimits preserve levelwise trivial
fibrations and fibrations.
\end{proof}

\section{Homotopy model structure}\label{homotopy-model}
The main objective of our work is to classify homotopy functors
from spectra to spectra, up to homotopy. The most convenient way
to provide such a classification is to define a model category
structure on small functors with fibrant objects being exactly the
fibrant homotopy functors, and to find a more familiar Quillen
equivalent model.

In this section, we define the homotopy model structure using the
extension of the Bousfield-Friedlander localization technique,
\cite[Appendix A]{BF}, to non-functorial localization
constructions, \cite[Appendix A]{Duality}.

We start from the fibrant-projective model structure on the
category of small functors (i.e., weak equivalences are levelwise
in fibrant objects). Homotopy functors are small functors that
preserve weak equivalences of fibrant objects. If we precompose a
homotopy functor with a fibrant replacement in $\Sp$, we obtain a
homotopy functor in the classical sense (preserving all weak
equivalences), which is fibrant-projective equivalent to the
original functor.

Fibrant homotopy functors are the local objects with respect to the following class of maps:
\[
\cal H = \{R^{B}\to R^{A} \,|\, A\we B \text{ weak equivalence of fibrant objects in } \Sp\}.
\]

Recall that the class of generating trivial cofibrations for the fibrant-projective model structure is

\[
\cal J = \{R^{A}\otimes K\cofib R^{A}\otimes L \,|\, A\in \Sp \text{ fibrant; } K\trivcofib L \text{ generating triv. cofibration in } \Sp\}.
\]

\subsection{Construction of homotopy localization}
We formulate our construction and argumentation in such a way that
it will be evident that the category of spectra may be replaced by
any closed symmetric monoidal combinatorial model category.

If we were able to localize with respect to the proper class of
maps $\cal H$, we would be done, since $\cal H$-local functors are
exactly the homotopy functors. Instead, for each particular
functor $F\in \Sp^\Sp$ we choose a cardinal $\lambda_F$, which is
the maximum between the accessibility rank of the small (i.e.,
accessible) functor $F$ and the degree of accessibility of the
subcategory of weak equivalences in spectra. Then, we localize
this particular functor $F$ with respect to a \emph{set} of maps
$\cal H_{\lambda_F}\subset \cal H$, and argue that, for this
specific functor $F$, it is enough to invert the set $\cal
H_{\lambda_F}$. Of course, we do not obtain a functorial
localization construction in this way. However, the
(non-functorial) localization we do obtain has enough good
properties to ensure the existence of the localized model
structure. The detailed construction follows.

\begin{definition}
Let $F\in \Sp^\Sp$ be a small functor of accessibility rank $\mu$
and let $\Sp$ be a $\kappa$-combinatorial closed symmetric
monoidal model for spectra. In particular, the domains and the
codomains of the generating (trivial) cofibrations are
$\kappa$-presentable, and the class of weak equivalences is a
$\kappa$-accessible subcategory of the category of maps of
spectra.  Put $\lambda_F = \max\{\kappa, \mu\}^+\rhd \max\{\kappa,
\mu\}$ (the $+$ is essential to ensure that the subcategory of
weak equivalences in $\Sp$ is still $\lambda_F$-accessible),
$\Sp_{\lambda_F}\subset \Sp$, the subcategory of
$\lambda_F$-presentable objects. Then, we define
\[
\cal H_{\lambda_F} = \{R^{B}\to R^{A} \,|\, A\we B \text{ weak equivalence of fibrant objects in } \Sp_{\lambda_F}\}
\]
and
\[
\cal J_{\lambda_F} = \{R^{A}\otimes K\cofib R^{A}\otimes L \,|\, A\in \Sp_{\lambda_F} \text{ fibrant; } K\trivcofib L \text{ generating triv. cofibration in } \Sp\}.
\]

As usual, we say that a map $f\colon F\to G$ is an $\cal H_{\lambda_F}$-equivalence if for every cofibrant
replacement $\tilde f\to f$ and every  $\cal H_{\lambda_F}$-local functor $W$ the induced map $\hom(\tilde f, W)$
is  a weak equivalence of simplicial sets.
\end{definition}

\begin{remark}
\begin{enumerate}
\item $\cal H_{\lambda_F}$ and $\cal J_{\lambda_F}$ are sets of
maps, rather than proper classes, and hence it is possible to
apply the small object argument. \item Every $\cal H$-local
functor is also $\cal H_{\lambda_F}$-local, and hence every $\cal
H_{\lambda_F}$-equivalence is also an $\cal H$-equivalence. \item
Every $\lambda_F$-accessible functor taking fibrant values in
$\lambda_F$-presentable fibrant objects (i.e., satisfying the
right lifting property with respect to $\cal J_{\lambda_F}$) is
fibrant-projectively fibrant. \item Every $\cal
H_{\lambda_F}$-local functor that is also $\lambda_F$-accessible
is $\cal H$-local. \item Every $\cal H_{\lambda}$-equivalence of
$\lambda$-accessible functors is an $\cal H$-equivalence.
\end{enumerate}
\end{remark}

We form the set of horns on $\cal H_{\lambda_F}$ by first
replacing every map in $\cal H_{\lambda_F}$ with a cofibration,
obtaining the set $\tilde{\cal H}_{\lambda_F}$, and then forming a
box product with every generating cofibration in $\Sp$:
\[
\Hor(\cal H_{\lambda_F})=\{
A\otimes L \coprod_{A\otimes K} B\otimes K \to B\otimes L \,|\,
(A\cofib B) \in \tilde{\cal H}_{\lambda_F} \text{ and } K\cofib L \text{ a gen. cofib. in } \Sp
\}
\]

It is well known (see, e.g., \cite{Hirschhorn}) that if a
fibration $X\to \ast$ has the right lifting property with respect
to $\Hor(\cal H_{\lambda_F})$, $X$ is $\cal H_{\lambda_F}$-local,
and therefore, in order to construct a localization of an $F\in
\Sp^\Sp$ with respect to $\cal H_{\lambda_F}$, it suffices to
apply the small object argument for the map $F\to \ast$ with
respect to the set $\cal L= \Hor(\cal H_{\lambda_F})\cup \cal
J_{\lambda_F}$. We obtain a factorization $F\cofib Q(F)\fibr
\ast$, where the cofibration is an $\cal L$-cellular map and the
fibration has the right lifting property with respect to $\cal K$.

We omit the standard verification based on the left properness of $\Sp^\Sp$, \cite[Section~4]{Duality},
that the cofibration $F\cofib QF$ is an $\cal H_{\lambda_F}$-equivalence,
and conclude that $QF$ is a homotopy localization of $F$ with respect to $\cal H_{\lambda_F}$.

Notice that $QF$ is obtained as a colimit of
$\lambda_F$-accessible functors, and therefore $QF$ is itself a
$\lambda_F$-accessible functor. However, the class of weak
equivalences in spectra is $\lambda_F$-accessible, and hence every
weak equivalence is a $\lambda_F$-filtered colimit of weak
equivalences between $\lambda_F$-presentable objects. Fibrant
objects in spectra are closed under $\lambda_F$-filtered colimits,
and every spectrum is a $\lambda_F$-filtered colimit of
$\lambda_F$-presentable spectra. Combining these facts with the
$\lambda_F$-accessibility of $QF$, we conclude that $QF$ is a
fibrant homotopy functor in the fibrant-projective model
structure. Moreover, the cofibration $\eta_F\colon F\cofib QF$ is
an $\cal H$-equivalence, since any $\cal
H_{\lambda_F}$-equivalence is an $\cal H$-equivalence. In other
words, we have constructed a homotopy localization of $F$ with
respect to the class $\cal H$ of maps. The only disadvantage of
our construction is the lack of functoriality, since it depends on
the choice of the cardinal $\lambda_F$ specific for each $F$.

Since $Q$ is not a functor, we have to define separately its
action on maps. Given a natural transformation of functors
$f\colon F\to G$, we define $Qf$ as a lifting in the diagram
\[
\xymatrix{
F
\ar@{^(->}[d]_{\eta_F}
\ar[r]^f
      &  G
         \ar[r]^{\eta_G} & QG
                 \ar@{->>}[d]\\
QF
\ar[rr]
\ar@{-->}[urr]|{Qf}
       &  &  \ast
}
\]
The lift exists since the left vertical map is $\cal L$-cellular and the right vertical map
is $\cal L$-injective by construction.

Of course, we will have to choose $Qf$ out of many maps that are
homotopic to each other, but the important property satisfied by
any of these choices is the commutativity of the square

\begin{equation} \label{A2}
\xymatrix{
F
\ar@{^(->}[d]_{\eta_F}
\ar[r]^f
      &  G
         \ar@{^(->}[d]^{\eta_G} \\
QF
\ar[r]_{Qf}
       & QG.
}
\end{equation}

\begin{proposition}\label{H-eq}
Let $f\colon F\to G$ be  a map of two functors. Then, $Qf$ is a
weak equivalence iff $f$ is an $\cal H$-equivalence.
\end{proposition}
\begin{proof}
The if direction follows by the `2-out-of-3' property for $\cal
H$-equivalences applied to a commutative square (\ref{A2}) and the
$\cal H$-local Whitehead theorem (cf., \cite[3.2.13]{Hirschhorn}):
an $\cal H$-local equivalence of $\cal H$-local functors is a weak
equivalence.

The only if direction follows since, if $Qf$ is a weak
equivalence, $f$ is an $\cal
H_{\max\{\lambda_F,\lambda_G\}}$-equivalence by  the `2-out-of-3'
property for $\cal H_{\max\{\lambda_F,\lambda_G\}}$-equivalences,
but $\cal H_{\max\{\lambda_F,\lambda_G\}}$-equivalence of
$\max\{\lambda_F,\lambda_G\}$-accessible functors is an $\cal
H$-equivalence.
\end{proof}

\subsection{Localization of the model structure}\label{verification}
The lack of functoriality of the homotopy localization $Q$ does
not allow us to apply Bousfield-Friedlander localization
machinery, \cite[Appendix A]{BF}. Instead, we will use the
generalization of their localization theorem developed in
\cite[Appendix A]{Duality}.

In order to apply this generalization of Bousfield-Friedlander
machinery, we need to verify a number of properties of the
localization construction $Q$.

The property \cite[A2]{Duality} requires precisely the commutativity of the diagram \ref{A2}, which we obtained by
construction.

The properties \cite[A3,A4]{Duality} are satisfied, since the
class of weak equivalences is defined as a map that, after a
cofibrant replacement, induces a weak equivalence on the mapping
spaces into every $\cal H$-local object $W$, since mapping out of
a retract diagram produces a retract diagram, and also any
commutative triangular diagram gives rise to a commutative
triangular diagram, which allows us to verify the 2-out-of-3
property.

In order to verify \cite[A5]{Duality}, for every commutative square
\begin{equation}\label{square}
\xymatrix{
F_1
\ar[r]
\ar[d] & F_2
         \ar[d]\\
F_3
\ar[r] & F_4
}
\end{equation}
Let $\lambda = \max\{\lambda_{F_i}\}^+_{1\leq i \leq 4} \rhd
\max\{\lambda_{F_i}\}_{1\leq i \leq 4}$ be a cardinal and
construct $Q'F_i$ exactly as $QF_i$  using only the cardinal
$\lambda$ instead of $\lambda_{F_i}$ for each $1\leq i\leq 4$.
Then, for all $1\leq i\leq 4$, there exists a factorization of the
coaugmentation map $\eta'_{F_i}\colon F_i\to Q'F_i$ as follows:
\[
\xymatrix{
F_i
\ar@{^(->}[r]^{\eta_{F_i}}
\ar@/_20pt/@{^(->}[rr]_{\eta'_{F_i}}
                & QF_i
                \ar[r]^{\dir{~}} & Q'F_i.
}
\]

Moreover, since the only obstruction for naturality of the
construction $Q$ is the choice of a different cardinal $\lambda_F$
for each functor $F$, here this obstruction is removed, and we
obtain a natural map of the diagram (\ref{square}) into the
commutative square
\[
\xymatrix{
QF_1
\ar[r]
\ar[d] & QF_2
         \ar[d]\\
QF_3
\ar[r] & QF_4
}
\]
giving rise to a commutative cube.

An additional verification is required in order for
\cite[A5]{Duality} to be satisfied: $Q'f$ must be a weak
equivalence iff $Qf$ is. By Proposition~\ref{H-eq}, it suffices to
show that $Q'f$ is  a weak equivalence  iff $f$ is an $\cal
H$-equivalence. Similar to the commutative square (\ref{A2}), we
have a commutative square
\[
\xymatrix{
F
\ar@{^(->}[d]_{\eta'_F}
\ar[r]^f
      &  G
         \ar@{^(->}[d]^{\eta'_G} \\
Q'F
\ar[r]_{Q'f}
       & Q'G.
}
\]
The vertical arrows are $\cal H$-equivalences by construction, and
hence the `2-out-of-3' property for $\cal H$-equivalences implies
that $f$ is an $\cal H$-equivalence if and only if $Q'f$ is.

The last property of the homotopical localization $Q$ that
requires verification in order to conclude that there exists a
$Q$-localized model structure on $\Sp^\Sp$ is \cite[A6]{Duality}:
for all pullback squares
\begin{equation}\label{A6-square}
\xymatrix{
W
\ar[d]_{g}
\ar[r]
      &  X
         \ar[d]^{f} \\
Y
\ar@{->>}[r]_{h}
       & Z,
}
\end{equation}
where $h$ is an $\cal H$-fibration (i.e., it has the right lifting
property with respect to all $\cal H$-equivalences, which are also
cofibrations) and $f$ is an $\cal H$-equivalence; also, $g$ is an
$\cal H$-equivalence.

Unfortunately, we do not have a simple description of $\cal
H$-equivalences (apart from the fact that they coincide with the
$Q$-equivalences, i.e., with the maps converted to weak
equivalences by $Q$-construction), and therefore we will use the
properties of the stable model category satisfied by $\Sp^\Sp$.
Namely, we will use the fact that every homotopy pullback is a
homotopy pushout in the fibrant-projective model structure.

We start by replacing the commutative square \ref{A6-square} with
a weakly equivalent commutative square of cofibrant functors. If
we start from $\tilde W\trivfibr W$ and continue to factor maps
$\tilde W\trivfibr W \to Z$ and $\tilde W\trivfibr W \to X$ to
obtain $\tilde Z$ and $\tilde X$, respectively, there are two
possible ways to replace $Y$ by factoring $\tilde Z\trivfibr Z
\fibr Y$ or $\tilde X\trivfibr X \to Y$ to obtain two different
approximations, $\tilde Y_Z\trivfibr Y$ and $\tilde Y_X\trivfibr
Y$, respectively. Since the original square (\ref{A6-square}) is a
levelwise homotopy pullback square for values of the functor in
each fibrant spectrum, the outer square
\[
\xymatrix{
\tilde W
\ar@{^(->}[dd]
\ar@{^(->}[rr]
                & & \tilde X
                    \ar@{^(->}[d]
                    \ar@{-->}[dl]\\
                & \tilde Y
                  \ar@{-->}[dr]^{\dir{~}}
                           & \tilde Y_X
                             \ar@{->>}[d]^{\dir{~}}\\
\tilde Z
\ar@{^(->}[r]
\ar@{-->}[ur]
                & \tilde Y_Z
                  \ar@{->>}[r]^{\dir{~}} & Y
}
\]
is a fibrant-levelwise homotopy pullback and homotopy pushout square.
Put $\tilde Y= \tilde X\coprod_{\tilde W} \tilde Z \we Y$ to obtain a cofibrant approximation of the
original commutative square.

In order to verify whether $g$ is an $\cal H$-equivalence for each
$\cal H$-local (i.e., fibrant homotopy) functor $H$, we form a
commutative square of mapping spectra:
\[
\xymatrix{
H^{\tilde Y}
\ar@{->>}[d]^{\dir{~}}
\ar@{->>}[r]
                       & H^{\tilde Z}
                         \ar@{->>}[d]\\
H^{\tilde X}
\ar@{->>}[r] & H^{\tilde W},
}
\]
which is a homotopy pullback of spectra, and therefore also a
homotopy pushout of spectra.  Hence, the left properness of the
category of spectra implies that the right hand vertical map is a
weak equivalence of spectra. Therefore, the original map $g\colon
W\to Z$ is an $\cal H$-equivalence.

We conclude that by \cite[Theorem~A8]{Duality} there exists a $Q$-localization of the model structure,
i.e., this is a localization with respect to $\cal H$.

We finish this section with an extension of Proposition~\ref{Quillen-map} to the localized model structure.

\begin{proposition}
The adjunction constructed in Proposition~\ref{adjunction} is
still a Quillen pair after the localization, i.e., if we consider
the $Q$-local model structure on $\Sp^\Sp$ and the strict model
structure on $\text{pro-}\Sp$, the adjunction $(O,P)$ is a Quillen
adjunction.
\end{proposition}
\begin{proof}
By Dugger's lemma \cite[8.5.4]{Hirschhorn}, it suffices to check
that the right adjoint $P$ preserves fibrations between fibrant
objects and all trivial fibrations.

Trivial fibrations did not change after the localization, and
therefore it suffices to show that $P$ preserves fibrations of
fibrant objects. Let $f{^\op}\colon X_\bullet \to Y_\bullet$ be a
fibration of fibrant objects in $(\text{pro-}\Sp)^{\op}$. Then,
$f\colon Y_\bullet \to X_\bullet$ is an essentially levelwise
cofibration of essentially levelwise cofibrant objects (i.e., up
to reindexing). Choose a representative for $f$, which is a
commutative diagram of cofibrations between cofibrant spectra and
apply $P$. For such a representative, $Pf$ is a filtered colimit
of fibrations of functors represented in cofibrant spectra, i.e.,
a filtered colimit of projectively fibrant functors preserving
weak equivalences of fibrant objects. In other words, $Pf$ is a
fibration of homotopy functors. Homotopy functors are precisely
the $Q$-local functors, i.e., $Pf$ is a $Q$-fibration by
Lemma~\cite[A.10]{Duality}.
\end{proof}

The rest of the paper is devoted to the proof that this Quillen map is indeed a Quillen equivalence.
In other words, we will show that for every cofibrant functor $F\in \Sp^\Sp$,
every fibrant $X_\bullet \in (\text{pro-}\Sp)^{\op}$, and every map $f\colon F\to PX_\bullet$ in $\Sp^\Sp$,
the map $f$ is a $Q$-equivalence if and only if the corresponding map $f^\sharp\colon OF\to X_\bullet$
is a weak equivalence in $(\text{pro-}\Sp)^{\op}$.

\section{All homotopy functors are linear}\label{section-homotopy-linear}
The first reduction in our classification problem is to show that every small homotopy functor is linear,
i.e., that they take homotopy pushouts to homotopy pullbacks. We shall classify the linear functors in the next section.

\begin{proposition}\label{homotopy-linear}
Every small functor $F\in \Sp^\Sp$ taking weak equivalences to weak equivalences also takes homotopy pushouts
to homotopy pullbacks.
\end{proposition}
\begin{remark}
This phenomenon appears only for functors enriched over spectra, such as the small functors,
which are colimits of representables. There is a fully featured calculus theory for \emph{simplicial}
functors from spectra to spectra developed by Michael Ching, \cite{Ching}, where the $n$-excisive functors
appear for every $n$.
\end{remark}
\begin{proof}
Given a small homotopy functor $F$, consider its  cofibrant
replacement $\tilde F$, which is a cellular functor and has the
filtration
\[
0 = F_{0}\cofib \ldots F_{n}\cofib F_{n+1}\cofib \ldots F_{\lambda} = \tilde F,
\]
where $F_{n+1}$ is obtained from $F_{n}$ by attaching a cell:
\begin{equation}\label{attach-cell}
\xymatrix{
R^{A}\wedge K
\ar[r]
\ar@{^(->}[d]
                & F_n
                    \ar@{^(->}[d]\\
R^{A}\wedge L
\ar[r]
                & F_{n+1},
}
\end{equation}
where $A$ may be chosen to be a cofibrant and fibrant spectrum, \cite[Prop.~5.3]{Duality},
and $K\cofib L$ a generating cofibration in spectra.

Our first goal is to show that if $F_{n}$ in the diagram
(\ref{attach-cell}) preserves homotopy pullbacks (which coincide
with homotopy pushouts) of fibrant spectra, $F_{n+1}$  also
preserves homotopy pullbacks of fibrant spectra. It will give an
inductive step.

Notice that all three functors in the commutative square
(\ref{attach-cell}) preserve homotopy pullbacks of fibrant
spectra. We will show that $F_{n+1}$ preserves homotopy pushouts,
too. Since homotopy pullbacks are also homotopy pushouts, this is
a rather intuitive statement of the kind ``a homotopy pushout of
homotopy pushouts is a homotopy pushout again''. The formal
argument will say that homotopy colimits commute with homotopy
colimits, and hence if we apply the functors in
(\ref{attach-cell}) on a homotopy pushout square, we obtain a
diagram over the category $\cat K \times \cat K$, where \cat K is
the category
\[
\xymatrix{
\bullet
\ar[r]
\ar[d]
                & \bullet
                    \ar[d]\\
\bullet
\ar[r]
                & \bullet \ar@{}[r]_>>>>>>>>{,}&{}
}
\]
and conclude that the application of $F_{n+1}$ on any homotopy
pullback of fibrant spectra is a homotopy pullback again.

$\tilde F$ is a sequential homotopy colimit of functors preserving
homotopy pullbacks of fibrant objects, and therefore $\tilde F$
also preserves homotopy pullbacks of fibrant objects. However, in
addition, $\tilde F$ is a homotopy functor, and hence it is a
linear functor, which is fibrant-projective equivalent to the
original small homotopy functor $F$.
\end{proof}

\section{Classification of small linear functors}\label{section-classification}

In this section, we present a classification of small linear
functors. These are the small functors taking homotopy pushouts
(=homotopy pullbacks) to homotopy pullbacks. Note that, since
every small functor $F\in \Sp^{\Sp}$ is a weighted colimit of
representable functors, it preserves the zero spectrum up to
homotopy. Note also that every linear functors is a homotopy
functor.

Let $\cal F$ be the class of maps ensuring that $\cal F$-local objects are precisely the fibrant linear functors.
Namely,
\[
\cal F=\left\{\left.
\hocolim\left(
\vcenter{
\xymatrix@=10pt{
R^{D}
\ar[r]
\ar[d]  & R^{B}\\
R^{C}
}
}
\right)
\longrightarrow R^{A}
\right|
\vcenter{
\xymatrix@=10pt{
A
\ar[r]
\ar[d] & B
             \ar[d]\\
C
\ar[r]  &  D
}
}
\text{-- homotopy pullback in}\, \Sp
\right\}.
\]

Our goal is to show that every linear functor is (fibrant-projectively) weakly equivalent to a filtered
colimit of functors represented in cofibrant objects, i.e., to an image of a cofibrant pro-spectrum under the restricted
Yoneda embedding $P$ constructed in Section~\ref{preliminaries}. We begin with the lemma stating that these functors
are closed under filtered colimits. In other words, filtered colimits of filtered colimits of representable
functors are again filtered colimits.

\begin{lemma}\label{filtered-colimit}
The full subcategory generated by the filtered colimits of representable functors is closed under filtered colimits.
Moreover, the subcategory of filtered colimits of functors represented in cofibrant objects is also closed
under filtered colimits.
\end{lemma}
\begin{proof}
Let $\cal I$ be a filtered category, and for each $i\in \cal I$
let $\cal J_i$ be a filtered category. Suppose that
$F_i=\colim_{j\in \cal J_i} R^{X_{i,j}}$ for some $X_{i,j}\in
\Sp$, and $F=\colim_{i\in \cal I} F_i$. Then, we need to show that
$F$ may be represented as a filtered colimit of representable
functors.

Applying the left adjoint $O$ on the functor $F$, we obtain a
pro-object $\{X_j\}_{j\in \cal J}$ for some filtered category
$\cal J$.
\[
\{X_\bullet\} = O(F)= O(\colim_{i\in \cal I} F_i) = {\colim_{i\in \cal I}}^{(\text{pro-}\Sp)^{\op}} O(F_i) = {\lim_{i\in \cal I}}^{\text{pro-}\Sp}\{X_{i,\bullet}\}.
\]

However, if we apply $P$ on $\{X_\bullet\}$, we recover $F$ again:

\begin{multline*}
P(\{X_\bullet\})(-)= \hom_{\text{pro-}\Sp}(\{X_\bullet\},-) \\
\shoveright{= \hom_{\text{pro-}\Sp}({\lim_{i\in \cal I}}^{\text{pro-}\Sp} \{X_{i,\bullet}\},-) \qquad \text{(constant pro-spaces are co-small)}}\\
 =\colim_{i\in \cal I}(\colim_{j\in \cal J_i} R^{X_{i,j}})= \colim_{i\in \cal I} F_i = F.
\end{multline*}

Therefore, $F = P(\{X_\bullet\}) = \colim_{j\in \cal J} R^{X_j}$ is a filtered colimit of representable functors.

Suppose now that all $X_{i,j}\in \Sp, \, i\in \cal I, \, j\in \cal
J_i$ are cofibrant spectra. Then, $O(F)=\{X_\bullet\}$ is a
cofibrant pro-spectrum as a cofiltered inverse limit of cofibrant
pro-spectra $\{X_{i,\bullet}\}$ in the class-fibrantly generated
strict model structure on pro-spectra \cite{pro-spaces}. In other
words, $\{X_\bullet\}$ is an essentially levelwise cofibrant
pro-spectrum, and hence $P(\{X_\bullet\})$ is a filtered colimit
of functors represented in cofibrant spectra.
\end{proof}

\begin{proposition}\label{F-euiv-filtered}
Let $F\in \Sp^{\Sp}$ be a linear functor. Then there exists a
filtered diagram $J$ and a functor $G=\colim_{j\in J}R^{X_j}$ with
cofibrant $X_j\in \Sp$ for all $j\in J$ and a weak equivalence
$f\colon \tilde F\to G$ for some cellular approximation $\tilde
F\we F$ in the fibrant-projective model structure.
\end{proposition}
\begin{proof}
Since $F$ is a linear functor, it is also a homotopy functor, and
hence there exists a cellular approximation $\tilde F\we F$ such
that for some cardinal $\lambda$ there is a transfinite sequence
of functors $\tilde F = \colim_{i\leq \lambda}F_{i}$, and $F_{i}$
is obtained from $F_{i-1}$ by attaching a generating cofibration
of the form $A\wedge R^{\hat X}\cofib B\wedge R^{\hat X}$ for
every successor cardinal $i\leq \lambda$ and
$F_{i}=\colim_{a<i}F_{a}$ for every limit ordinal $i\leq \lambda$.
The cofibration $A\cofib B$ is a generating cofibration in $\Sp$,
and therefore, $A$ and $B$ are compact spectra. Moreover, the
representing object $\hat X$ may be chosen to be cofibrant, since
$\tilde F$ is a homotopy functor by \cite[Prop.~5.3]{Duality}.

By \cite[Lemma~3.3]{Chorny-Brownrep}, there exists a countable
sequence $\{F_{k}'\}_{k<\omega}$ such that $F'_{0}=0$,
$F=\colim_{k<\omega}F_{k}'$ and for each $k>0$ there is a pushout
square
\[
\xymatrix{
\displaystyle{\coprod_{s\in S_{k-1}} A_{s}\wedge R^{\hat X_{s}} }
\ar[r]
\ar@{^{(}->}[d]         &F'_{k-1}
                    \ar[d]\\
\displaystyle{\coprod_{s\in S_{k-1}} B_{s}\wedge R^{\hat X_s}}
\ar[r]           &F'_{k},\\
}
\]
where the coproduct is indexed by the subset $S_{k-1}\subset \lambda$ corresponding to the cells coming
from various stages of the original sequence $\{F_{i}\}_{i\leq \lambda}$, such that their attachment maps
factor through the $(k-1)$-st stage of the previously constructed sequence.

The coproduct of maps in the commutative square above is a
filtered colimit of finite coproducts over the filtering $J_{k-1}$
of the finite subsets of $S_{k-1}$. Let us think of the constant
object $F'_{k-1}$ as a filtered colimit of the constant diagrams
over the same filtering $J_{k-1}$. However, colimits over
$J_{k-1}$ commute with pushouts, and hence we obtain the
representation of $F'_{k}$ as a filtered colimit of pushouts of
the following form
\begin{equation}\label{pushout-square}
\xymatrix{
\displaystyle{\coprod_{s\in S_{k-1,j}} A_{s}\wedge R^{\hat X_{s}} }
\ar[r]_<<<<{\varphi_{k-1,j}}
\ar@{^{(}->}[d]         &F'_{k-1}
                    \ar[d]\\
\displaystyle{\coprod_{s\in S_{k-1,j}} B_{s}\wedge R^{\hat X_s}}
\ar[r]           &F_{k,j},\\
}
\end{equation}
where $S_{k-1,j}\subset S_{k-1}$ is a finite subset corresponding to the element $j\in J_{k-1}$.

Now, by \cite[Lemma~7.1]{Duality}, there are weak equivalences in
the fibrant projective model category: $A_{s}\wedge
R^{\hat{X}_{s}}\simeq R^{\hat{X}_{s}\wedge DA_{s}}$ and
$B_{s}\wedge R^{\hat{X}_{s}}\simeq R^{\hat{X}_{s}\wedge DB_{s}}$.
Moreover, any finite coproduct of representable functors is $\cal
F$-equivalent to a representable functor by an inductive argument
on the number of terms that begins with an observation that a
coproduct of two representables $R^{\hat U}\sqcup R^{\hat V}$ is
$\cal F$-equivalent to $R^{\hat U\times \hat V}$, since the map
$R^{\hat U}\sqcup R^{\hat V}\simeq\hocolim(R^{\hat U}\leftarrow
R^{0}\to R^{\hat V})\longrightarrow R^{\hat U\times\hat V}$ is an
element in $\cal F$ corresponding  to the homotopy pullback square
\[
\xymatrix{
\hat U\times\hat V
\ar[r]
\ar[d]                     & \hat U
                                 \ar[d]\\
\hat V
\ar[r]                    &  0.\\
}
\]

In other words, the entries on the left-hand side of the push-out square (\ref{pushout-square})
are $\cal F$-equivalent to representable functors with fibrant and cofibrant spectra as representing objects.

Suppose for induction that there is an $\cal F$-equivalence
$F'_{k-1}\to \colim_{l\in L_{k-1}}R^{Y_l}$, where $L_{k-1}$ is a
filtered category and the representable functors have fibrant and
cofibrant spectra as representing objects. Then, we obtain a
morphism of the pushout diagram (\ref{pushout-square}) into a
commutative square (which is also a homotopy pushout) composed of
filtered colimits of representable functors constructed as follows
\begin{equation}\label{dash-arrow}
\xymatrix{
R^{\left(\displaystyle{\prod_{s\in S_{k-1,j}}}\hom(A_s, \hat{X_s})\right)_{\text{cof}}}
\ar[rrr]_{\varphi_{k-1,j}}
\ar[ddd]        &  &  & \displaystyle{\colim_{l\in L'_{k-1}}} R^{Y_l}
                        \ar[ddd]\\
 & \displaystyle{\coprod_{s\in S_{k-1,j}} A_{s}\wedge R^{\hat X_{s}}}
\ar@{^{(}->}[d]
\ar[r]
\ar[ul]             &F'_{k-1}
                    \ar[d]
                    \ar[ur]\\
 & \displaystyle{\coprod_{s\in S_{k-1,j}} B_{s}\wedge R^{\hat X_s}}
    \ar[r]
    \ar[dl]          & F_{k,j}
                        \ar@{-->}[dr]\\
R^{\left(\displaystyle{\prod_{s\in S_{k-1,j}}}\hom(B_s, \hat{X_s})\right)_{\text{cof}}}
\ar[rrr]        &  &  & \displaystyle{\colim_{l\in L'_{k-1}}} R^{Y'_l}.\\
}
\end{equation}
The diagonal maps on the left are obtained as compositions of the unit of the adjunction (\ref{adjunction})
with a map induced by the cofibrant approximations in pro-$\Sp$:
\begin{eqnarray*}
\left(\displaystyle{\prod_{s\in S_{k-1,j}}}\hom(A_s, \hat{X_s})\right)_{\text{cof}} \trivfibr \displaystyle{\prod_{s\in S_{k-1,j}}}\hom(A_s, \hat{X_s}),\\
\left(\displaystyle{\prod_{s\in S_{k-1,j}}}\hom(B_s, \hat{X_s})\right)_{\text{cof}} \trivfibr \displaystyle{\prod_{s\in S_{k-1,j}}}\hom(B_s, \hat{X_s}).
\end{eqnarray*}

The universal property of the unit of adjunction guarantees the existence of a natural map
\[
R^{\displaystyle{\prod_{s\in S_{k-1,j}}}\hom(A_s, \hat{X_s})} \longrightarrow \colim_{l\in L_{k-1}}R^{Y_l}.
\]
The corresponding map in the pro-category has a lift to the cofibrant replacement of the constant pro-spectrum,
since the pro-spectrum $\{Y_l\}_{l\in L_{k-1}}$ is (levelwise) cofibrant.
\[
\xymatrix{
 & \left(\displaystyle{\prod_{s\in S_{k-1,j}}}\hom(A_s, \hat{X_s})\right)_{\text{cof}}
    \ar@{->>}[d]^{\dir{~}} \\
 \{Y_l\}_{l\in L_{k-1}}
 \ar[r]
 \ar@{-->}[ur] & \displaystyle{\prod_{s\in S_{k-1,j}}}\hom(A_s, \hat{X_s})
}
\]
The source of the dashed map in the diagram above may be replaced
by an isomorphic pro-object $\{Y_l\}_{l\in L'_{k-1}}$ with a final
indexing subcategory $L'_{k-1}\subset L_{k-1}$, so that the
resulting map is reindexed into a natural transformation of
contravariant $L'_{k-1}$-diagrams with a constant diagram in the
target. The induced map in the category of functors is denoted by
$\varphi_{k-1,j}$, and it factors through every stage of the
colimit. Thus, the outer pushout diagram in (\ref{dash-arrow}) may
be viewed as a filtered colimit of pushout diagrams indexed by
$L'_{k-1}$.

Let $Y'_l = Y_l \times_{\left(\displaystyle{\prod_{s\in
S_{k-1,j}}}\hom(A_s, \hat{X_s})\right)_{\text{cof}}}
\left(\displaystyle{\prod_{s\in S_{k-1,j}}}\hom(B_s,
\hat{X_s})\right)_{\text{cof}}$. This is a homotopy pullback of
spectra, and hence $R^{Y'_l}$ is $\cal F$-equivalent to the
homotopy pushout of the corresponding representable functors in
the fibrant-projective model structure on the category of small
functors from spectra to spectra.

Taking the filtered colimit of these commutative squares indexed
by $L'_{k-1}$, we obtain the outer square of (\ref{dash-arrow}),
and since filtered colimits preserve both $\cal F$-equivalences by
\cite[Lemma~1.2]{CaCho} and homotopy pushouts, we conclude that
$\displaystyle{\colim_{l\in L'_{k-1}}} R^{Y'_l}$ is $\cal F$ to
the homotopy pushout of the outer square of (\ref{dash-arrow}).
Therefore, the dashed arrow in (\ref{dash-arrow}) is an $\cal
F$-equivalence. In other words, $F_{k,j}$ is $\cal F$-equivalent
to a filtered colimit of representable functors.

Therefore, $F'_{k}=\colim_{j\in J_{k-1}} F_{k,j}$ is a filtered colimit of functors $\cal F$-equivalent
to filtered colimits of representable functors, which, in turn, are $\cal F$-equivalent to filtered colimits
of representable functors by Lemma~\ref{filtered-colimit}.

Finally, $F=\colim_{k<\omega}F'_{k}$ is a countable sequential colimit of filtered colimits
of functors $\cal F$-equivalent to representable functors, which may be reindexed into a single filtered colimit
of functors $\cal F$-equivalent to representable functors by Lemma~\ref{filtered-colimit}.
\end{proof}

\begin{corollary}\label{classification}
Every small homotopy functor from spectra to spectra is fibrant-projective equivalent to a filtered colimit
of representable functors represented in cofibrant objects.
\end{corollary}
\begin{proof}
Every small homotopy functor $F\in \Sp^\Sp$ is linear by
Proposition~\ref{homotopy-linear}. Therefore, $F$ is $\cal
F$-local and, by Proposition~\ref{F-euiv-filtered}, is $\cal
F$-equivalent to a filtered colimit of representable functors
represented in cofibrant objects. However, $\cal F$-equivalence of
$\cal F$-local functors is a fibrant-projective equivalence.
\end{proof}

So far, we have shown that the fibrant objects in the homotopy
model structure constructed in Section~\ref{homotopy-model} are
fibrant projective equivalent to filtered colimits of
representable functors represented in cofibrant objects, i.e.,
they correspond to cofibrant pro-objects. Of course, a more
elegant way to state this classification result is to show that
the Quillen adjunction of Proposition~\ref{Quillen-map} is
actually a Quillen equivalence.

\begin{theorem}\label{main-theorem}
The Quillen adjunction
$\xymatrix{O\colon \Sp^\Sp
                \ar@/_/[r]  &
                                \proSp \!: P
                                \ar@/_/[l]}$
is a Quillen equivalence if $\Sp^\Sp$ is equipped with the
homotopy model structure and $\proSp$ is equipped with the strict
model structure.
\end{theorem}

The rest of the paper is devoted to the proof of this theorem.

\section{Alternative localization construction}\label{section-localization-construction}

In this section, we give an alternative localization of the
fibrant-projective model structure on $\Sp^\Sp$, which produces
homotopy approximations of small functors. It is better suited for
establishing that the Quillen map constructed in
Proposition~\ref{adjunction} is a Quillen equivalence.

\subsection{The localization construction}\label{localization-construction}
Given an arbitrary small functor $F\in \Sp^{\Sp}$, consider its
(non-functorial) cofibrant replacement in the fibrant-projective
model structure $\tilde F \trivfibr F$. Then, the derived unit of
the adjunction constructed in Proposition~\ref{adjunction} has the
right homotopy type of the localization we are constructing:
$u\colon \tilde F\to P(\widehat {O(\tilde F)})$. However, the
localization construction involves a coaugmentation map for every
functor $\eta: F\to LF$. Now, we factor $u$ into a cofibration
followed by a trivial fibration $\tilde F \cofib F_{1} \trivfibr
P\widehat{O\tilde F}$, and declare $LF= F\times_{\tilde F} F_{1}$.

We summarize our localization construction in the following
diagram:
\begin{equation}\label{L-constr}
\xymatrix{
\tilde F \ar@{->>}[dd]_{\dir{~}}
            \ar[rr]^{u_F}
            \ar@{^(->}[dr]  &           &{P\widehat{O\tilde F}}\\
                                    & F_{1} \ar@{->>}[ur]^{\dir{~}} \ar[d]\\
F \ar@{^(->}[r] & LF
}
\end{equation}

The construction of $LF$ depends on the choice of a cofibrant
replacement for $F$ and a factorization for $u_F$. We fix these
choices once and for all. Since the procedure described above is
homotopy meaningful, the homotopy type of $LF$ does not depend on
the choices we make.

The localization construction $L$ is defined also on morphisms.
Given a natural transformation $g\colon F\to G$ of small functors
from spectra to spectra, we proceed through the stages of the
definition of $L$, constructing at each stage a map corresponding
to $g$, making a choice in the non-functorial parts of the
definition of $L$. Namely, to obtain a map of cofibrant
replacements and the factorizations, we use the lifting axiom.
Such a choice is not functorial, and it is unique only up to
homotopy. Nevertheless, we have the following commutative diagram:

\[
\xymatrix{
P\widehat{O\tilde F}
\ar[rrr]^{P\widehat{Q\tilde g}}        &           &          & P\widehat{O\tilde G}\\
                                & F_{1}
                                  \ar@{->>}[ul]_{\dir{~}}
                                  \ar[r]^h
                                  \ar'[d][dd]       & G_{1} \ar@{->>}[ur]^{\dir{~}}
                                                    \ar'[d][dd]  \\
\tilde F
\ar@{->>}[dd]_{\dir{~}}
\ar[uu]^{u}
\ar[rrr]_{\tilde g}
\ar@{^(->}[ur]&       &      & \tilde G  \ar@{->>}[dd]_{\dir{~}}
                            \ar[uu]_{v}
                            \ar@{_(->}[ul]\\
            & LF \ar[r]_{Lg} & LG\\
F \ar[rrr]_{g}
   \ar@{^(->}[ur]^{\eta_{F}}  & & & G \ar@{_(->}[ul]_{\eta_{G}}
}
\]
This diagram defines a map $Lg$ for every natural transformation
$g$ and a morphism of maps $\eta_{g}\colon g\to Lg$.


We summarize this discussion in the following proposition.
\begin{proposition}\label{almost-natural}
For every natural transformation $g\colon F\to G$ of small
functors, the map $Lg\colon LF\to LG$ is defined and depends on
the choices required at various stages of its construction.
Moreover, there exist maps $\eta_{F}\colon F\to LF$ and
$\eta_{G}\colon G\to LG$ depending on the same choices and no
others, such that the square
\[
\xymatrix{
F
\ar[r]^{\eta_{F}}
\ar[d]_{g}
            & LF
            \ar[d]^{Lg} \\
G
\ar[r]_{\eta_{G}} & LG
}
\]
is commutative.
\end{proposition}

Our goal is to compare the localization construction $L$ with the
non-functorial localization $Q$ previously constructed in
Section~\ref{homotopy-model}. However, first we need to prove that
$L$ is a homotopy localization construction in accordance with
Definition~\cite[A1]{Duality} and to verify the conditions
\cite[A2--A6]{Duality}. Proposition~\ref{almost-natural} above
verified the condition \cite[A2]{Duality}.

\subsection{Verification of homotopy idempotency}

\begin{proposition}\label{homotopy-idemp}
For all $F\in\Sp^\Sp$, the maps $\eta_{LF},L\eta_F \colon LF\to LLF$ are weak equivalences.
\end{proposition}

We begin with a technical lemma about class-combinatorial model
categories generalizing similar results for combinatorial model
categories: weak equivalences are closed under $\lambda$-filtered
colimits \cite[7.3]{Dugger-generation}.

\begin{lemma}\label{filtered-colimits}
Let $\cat M$ be a class-cofibrantly generated model category with
$\lambda$-presentable domains and codomains of generating
(trivial) cofibrations. Then, $\lambda$-filtered colimits of
objects in $\cat M$ are homotopy colimits. In other words, every
levelwise weak equivalence of $\lambda$-filtered diagrams in $\cat
M$ induces a weak equivalences between their colimits.
\end{lemma}
\begin{proof}
Let $\cat A$ be a $\lambda$-filtered category, and $\dgrm X, \dgrm
Y\colon \cat A\to \cat M$ be two diagrams, and let $f\colon\dgrm
X\to \dgrm Y$ be a levelwise weak equivalence. Consider the
projective model structure on the category $\cat M^\cat A$. It may
be constructed by a straightforward generalization of
\cite[11.6]{Hirschhorn}. Now, we apply a cofibrant replacement in
the projective model structure to the map $f$:
\[
\xymatrix{
\tilde X \ar@{->>}[d]_{\dir{~}}\ar[r]^{\tilde f}& \tilde Y\ar@{->>}[d]^{\dir{~}}\\
X \ar[r]_f &  Y
}
\]

The functor $\colim \colon \cat M^\cat A \to \cat M$ is a left
Quillen functor if the domain category is equipped with the
projective model structure, and hence it preserves weak
equivalences of cofibrant objects. Moreover, this functor
preserves trivial fibrations, since the category $\cat M$ is
class-cofibrantly generated. Therefore, applying a colimit on the
commutative square above, we conclude, by the 2-out-of-the-3
property for weak equivalences, that $\colim f$ is a weak
equivalence.
\end{proof}

\begin{lemma}\label{derived-unit}
Let $X_\bullet\in \proSp$ be a cofibrant pro-spectrum. Then,
$PX_\bullet\in \Sp^\Sp$ is a filtered colimit of representable
functors and not necessarily cofibrant. Consider a cofibrant
replacement, $p\colon \widetilde{PX_\bullet} \trivfibr
PX_\bullet$. Then, the left adjoint $O$ preserves this weak
equivalence: $Op\colon O\widetilde{PX_\bullet} \we OPX_\bullet$.
\end{lemma}
\begin{proof}
As the left Quillen functor $O$ preserves weak equivalences
between cofibrant objects,  it suffices to prove that $O$ takes
into a weak equivalence some cofibrant approximation of $P\{X_i\}
= \colim_i R^{X_i}$. Consider the cofibrant approximation:
$q\colon \hocolim_i R^{\hat X_i} = \overline{PX_\bullet}\to
PX_\bullet$, where $q$ is induced by the fibrant-projective
cofibrant approximations $R^{\hat X_i}\we R^{X_i}$, while the maps
$X_i \trivcofib \hat X_i$ are the functorial fibrant
approximations in $\Sp$. $\overline{PX_\bullet}$ is cofibrant as a
homotopy colimit of a diagram with cofibrant entries (we assume
here that a homotopy colimit is defined as a coend with a
projectively cofibrant, contractible diagram of spaces, i.e., a
left Quillen functor preserving cofibrant objects).  By
Lemma~\ref{filtered-colimits}, the map $q$ is a weak equivalence,
since filtered colimits in the class-cofibrantly generated
fibrant-projective model structure on $\Sp^\Sp$ are homotopy
colimits.

$O$ preserves colimits and homotopy colimits as a left Quillen
functor, and hence the map $Oq\colon O\overline{PX_\bullet}\to
OPX_\bullet$ is essentially the map $Oq\colon \hocolim OR^{\hat
X_i}\to \colim OR^{X_i}$, or just $Oq\colon \hocolim \hat X_i \to
\colim X_i$ in the opposite of the strict model structure on
$(\proSp)^{\op}$. However, the strict model structure on $\proSp$
is class-fibrantly generated \cite{pro-spaces}, and therefore the
dual model structure is class-cofibrantly generated and the map
$Oq$ is a weak equivalence by Lemma~\ref{filtered-colimits}.
Therefore, $Op$ is also a weak equivalence.
\end{proof}

\begin{lemma}\label{O-eq}
The map $O \widetilde{\eta_F}\colon O\widetilde{F} \we O\widetilde{LF}$ is a  weak equivalence for all $F\in \Sp^{\Sp}$.
\end{lemma}
\begin{proof}
In the commutative diagram  (\ref{L-constr}), the object  $F_{1}$
may serve as a cofibrant replacement for both $LF$ and
$P\widehat{O\tilde F}$. Therefore, applying the functor $O$ on the
commutative diagram (\ref{L-constr}), we conclude that the map
$O\widetilde{\eta_F}\colon O\widetilde{F} \we O\widetilde{LF}$ is
a weak equivalence: the trivial fibration $F_{1}\trivfibr
P\widehat{O\tilde F}$ remains  a weak equivalence after
application of $O$ by Lemma~\ref{derived-unit}, and the derived
unit of the  $(O,P)$-adjunction $u_{F}$ is also turned by $O$ into
a weak equivalence $O\widetilde F \to OP\widehat{O\widetilde F}
\cong \widehat{O\widetilde F}$, and hence the map of $\widetilde
F\cofib F_1$ is turned by the application of $O$ into a weak
equivalence; the map $O\widetilde{\eta_{F}}$ is then a weak
equivalence by the `2-out-of-the-3' property.
\end{proof}

\begin{proof}[Proof of Prop.~\ref{homotopy-idemp}]
Consider the construction of $\eta_{LF}$ first.
\[
\xymatrix{
\tilde F
\ar@{->>}[dd]^{\dir{~}}
\ar@{^(->}[dr]_{a}
\ar[rr]^{u_F}       & & P\widehat{O(\tilde F)}
                \ar[r]^{\dir{~}}_{m}  & P\widehat{O}(\widetilde{P\widehat{O\tilde F}})
                                \ar[r]^{\dir{~}}_{g}
                                        & P\widehat{OP\widehat{O\tilde F}}
                                        \ar@{=}[d]  \\
        & F_1
           \ar@{->>}[ur]^{\dir{~}}
           \ar[d]^{\dir{~}}
           \ar[rr]^{\dir{~}}_k
           \ar[rd]^{\dir{~}}_{b}    &  & P\widehat{O F_1}
                                \ar[d]^{\dir{~}}
                            \ar[u]_{\dir{~}}   &   P\widehat{\widehat{O\tilde F}}\\
F
\ar[r]^{\eta_F}     &   LF
                \ar@{^(->}[dr]_{\eta_{LF}}^{\dir{~}} & \widetilde{LF}
                                            \ar@{->>}[l]_{\dir{~}}
                                            \ar[r]^{\dir{~}}_{u_{LF}}
                                            \ar@{^(->}[dr]^{\dir{~}}_{l} & P\widehat{O\widetilde{LF}}\\
  &  & LLF & F_2.
            \ar@{->>}[u]^{\dir{~}}
            \ar[l]_{\dir{~}}
}
\]
In the commutative diagram above, the map $g$ is a weak
equivalence by Lemma~\ref{derived-unit}, and therefore the map $m$
is also a weak equivalence. Applying consecutively  the
`2-out-of-3' property, we find that the maps $k$, $u_{LF}$, and
$l$ are weak equivalences. Therefore, $\eta_{LF}$ is a weak
equivalence by the `2-out-of-3' property again.

Now, consider the construction of $L\eta_{F}$.
\[
\xymatrix{
P\widehat{O\tilde F}
\ar[rrr]^{P\widehat{O\widetilde{\eta_F}}}        &           &          & P\widehat{O\widetilde{LF}}\\
                                & F_{1}
                                  \ar@{->>}[ul]_{\dir{~}}
                                  \ar[r]^h
                                  \ar'[d][dd]       & F_2 \ar@{->>}[ur]^{\dir{~}}
                                                    \ar'[d][dd]  \\
\tilde F
\ar@{->>}[dd]_{\dir{~}}
\ar[uu]^{u_{F}}
\ar[rrr]_{\widetilde{\eta_F} }
\ar@{^(->}[ur]&       &      & \widetilde{LF} \ar@{->>}[dd]_{\dir{~}}
                            \ar[uu]_{u_{LF}}
                            \ar@{_(->}[ul]\\
            & LF \ar[r]_{L\eta_F} & LLF\\
F \ar[rrr]_{\eta_F}
   \ar@{^(->}[ur]^{\eta_F}  & & & LF \ar@{_(->}[ul]_{\eta_{LF}}
}
\]
The map $P\widehat{O\widetilde{\eta_F}}$ is a weak equivalence as
an application of the right Quillen functor $P$ on the weak
equivalence, by Lemma~\ref{O-eq}, between fibrant objects
$\widehat{O\widetilde{\eta_F}}$.

The `2-out-of-the-3' property then implies that $h$ and $L\eta_F$ are weak equivalences as well.
\end{proof}

In addition, we notice that every stage in the construction of $L$
preserves weak equivalences, and therefore we readily obtain the
following
\begin{proposition}\label{homotopy-inv}
The localization construction $L$ preserves weak equivalences.
\end{proposition}

We are  now ready to compare the two non-functorial localization
constructions and prove that $QF$ is weakly equivalent to $LF$ for
all $F\in \Sp^\Sp$. We already know that the classes of $Q$-local
objects and $L$-local objects coincide: these are fibrant functors
weakly equivalent to filtered colimits of representable functors
with cofibrant representing objects. Ideologically, this should
imply the equivalence of localization constructions immediately.
However, the proof of a general statement of this kind is involved
and requires plenty of additional structure on the localization
constructions, which does not exist in our case (cf.,
\cite{CaCho},\cite{Farjoun-natural}). Therefore, we shall carry
out the proof in this particular situation.

\begin{lemma}\label{derived-unit1}
The derived unit map $u_F\colon \tilde{F}\to P\widehat{O\tilde F}$ is a $Q$-equivalence for all $F\in \Sp^{\Sp}$.
\end{lemma}
\begin{proof}
By Proposition~\ref{H-eq}, it suffices to check whether $u_F$ is
an $\cal H$-equivalence, i.e., it suffices to verify that
$\hom(\widetilde{u_F}, W)$ for any $Q$-local functor $W$. By
Corollary~\ref{classification}, $W$ is weakly equivalent to a
filtered colimit of representable functors represented in
cofibrant spectra, and hence $W\simeq PX_\bullet$ for some
cofibrant pro-spectrum $X_\bullet$.

By adjunction, the map $\hom(\widetilde{u_F}, PX_\bullet)$ is
naturally isomorphic to the map $\hom(O\widetilde{P\widehat{O
\tilde F}}, X_\bullet)\to \hom({O\tilde F, X_\bullet})$.

By Lemma~\ref{derived-unit}, $O\widetilde{P\widehat{O \tilde
F}}\simeq OP\widehat{O \tilde F}= \widehat{O \tilde F}$, showing
that the last map is a weak equivalence.
\end{proof}

\begin{proposition}
For all $F\in \Sp^\Sp$ there is a weak equivalence $QF\simeq LF$.
\end{proposition}
\begin{proof}
First, we notice that the coaugmentation map $\eta_F\colon F\to
LF$ is a $Q$-equivalence.

Given $F$, similarly to the verification of \cite[A5]{Duality} in
\ref{verification}, we choose a cardinal $\lambda$ big enough that
all entries of the commutative diagram (\ref{L-constr}) are
$\lambda$-accessible. Next, we apply a modification of $Q$, which
is functorial and on this particular diagram provides results
weakly equivalent to the application of $Q$, and conclude that the
map $\eta_F\colon F\cofib LF$ is a $Q$-equivalence if and only if
the derived unit map $u_F\colon \tilde{F}\to P\widehat{O\tilde F}$
is a $Q$-equivalence. Hence, by Lemma~\ref{derived-unit1},
$\eta_F\colon F\to LF$ is a $Q$-equivalence.

Consider now the following commutative diagram obtained by
application of the construction $Q$ on the coaugmentation
\[
\xymatrix{
F
\ar[r]^{\eta_F}
\ar@{^(->}[d]
           & LF
              \ar@{^(->}[r]^{\dir{~}}
              \ar@{^(->}[d]^{\dir{~}}
                               & \widehat{LF}
                                   \ar@{->>}[d]\\
QF
\ar[r]_{\dir{~}}
           & QLF
              \ar[r]
              \ar@{-->}[ur]
                             & \ast
}
\]
Since $\eta_{F}$ is a $Q$-equivalence, $Q\eta_{F}$ is a weak
equivalence in the fibrant projective model structure, and hence
we obtain a zig-zag weak equivalence $QF\simeq LF$ for all $F\in
\Sp^{\Sp}$.
\end{proof}

\section{Proof of Theorem~\ref{main-theorem}}\label{section-Quillen-equivalence}
In Proposition~\ref{Quillen-map}, we have shown that the
adjunction $\xymatrix{O\colon \Sp^\Sp
                \ar@/_/[r]  &
                                (\proSp)^{\op} \!: P
                                \ar@/_/[l]}$
is a Quillen pair. We need to show that for every cofibrant $F\in
\Sp^\Sp$ and every fibrant $X_\bullet \in (\proSp)^{\op}$ the map
$f\colon O(F)\to X_\bullet$ is a (strict) weak equivalence of
pro-spectra if and only if the map $g\colon F\to PX_\bullet$ is a
weak equivalence in the homotopy  model structure on $\Sp^\Sp$,
i.e., it is a $Q$-equivalence of small functors.

Suppose that $f\colon O(F)\to X_\bullet$ is a weak equivalence.
Applying a fibrant replacement on $O(F)$, we obtain a trivial
cofibration $j\colon OF \trivcofib \widehat{OF}$ and a
factorization of $f$ as $f=\hat f j$, where the lifting $\hat{f}$
exists since $X_\bullet$ is fibrant (in $\proSp^\op$). Moreover,
$\hat{f}$ is a weak equivalence of fibrant objects by the
`2-out-of-3' property. The adjoint map $g$ factors as a unit of
the adjunction $u\colon F\to POF$ composed with $Pf$: $g=P(f)u$,
but $Pf = P(\hat{f}j)=P(\hat{f})Pj$, and hence $g=P(\hat{f}) (P(j)
u)$. Now, $P(\hat{f})$ is a weak equivalence, since $P$ is a right
Quillen functor and preserves weak equivalences of fibrant
objects. The composed map $P(j) u$ is an $L$-equivalence by
Proposition~\ref{homotopy-idemp} and it is a $Q$-equivalence by
Lemma~\ref{derived-unit1}, which applies since $F$ is cofibrant.

Conversely, suppose that $g\colon F\to PX_{\bullet}$ is a weak
equivalence. Consider a cofibrant replacement $p\colon
\widetilde{PX_{\bullet}}\trivfibr PX_{\bullet}$. Then, there
exists a lift $\tilde g\colon F\to \widetilde{PX_{\bullet}}$ in
the homotopy model structure. Note that $\tilde g$ is a weak
equivalence of cofibrant objects by the `2-out-of-3' property,
since $g=p\tilde g$. The adjoint map $f\colon OF \to X_{\bullet}$
factors as $Og$ followed by the counit $c\colon OPX_{\bullet}\to
X_{\bullet}$, which is a natural isomorphism for all
$X_{\bullet}$. However, $Og=OpO\tilde g$, where $Op$ is a weak
equivalence by Lemma~\ref{derived-unit} and $O\tilde g$ is a weak
equivalence, since $O$ is a left Quillen functor. Hence, $f$ is a
weak equivalence.

\bibliographystyle{abbrv}

\normalsize
\baselineskip=17pt


\bibliography{Xbib}

\begin{thebibliography}{10}

\bibitem{Arone-Ching-classification}
G.~Arone and M.~Ching.
\newblock A classification of {T}aylor towers of functors of spaces and
  spectra.
\newblock {\em Adv. Math.}, 272:471--552, 2015.

\bibitem{Duality}
G.~Biedermann and B.~Chorny.
\newblock Duality and small functors.
\newblock {\em Algebraic and Geometric Topology}, 2015.
\newblock To appear.

\bibitem{BCR}
G.~Biedermann, B.~Chorny, and O.~R\"ondigs.
\newblock {Calculus of functors and model categories}.
\newblock {\em Adv. in Math.}, 214(1):92--115, 2007.

\bibitem{BF}
A.~Bousfield and E.~M. Friedlander.
\newblock {Homotopy theory of $\Gamma$-spaces, spectra, and bisimplicial sets}.
\newblock In {\em {Geometric Applications of Homotopy Theory II}}, number 658
  in {Lecture Notes in Mathematics}. Springer, 1978.

\bibitem{CaCho}
C.~Casacuberta and B.~Chorny.
\newblock The orthogonal subcategory problem in homotopy theory.
\newblock In {\em An alpine anthology of homotopy theory}, volume 399 of {\em
  Contemp. Math.}, pages 41--53. Amer. Math. Soc., Providence, RI, 2006.

\bibitem{Ching}
M.~Ching.
\newblock {A} chain rule for {G}oodwillie derivatives of functors from spectra
  to spectra.
\newblock {\em Trans. Amer. Math. Soc.}, 362:399--426, 2010.

\bibitem{Chorny-ClassLinFun}
B.~Chorny.
\newblock A classification of small linear functors.
\newblock {\em Int. Math. Res. Not.}
\newblock To appear.

\bibitem{PhDI}
B.~Chorny.
\newblock Localization with respect to a class of maps. {I}. {E}quivariant
  localization of diagrams of spaces.
\newblock {\em Israel J. Math.}, 147:93--139, 2005.

\bibitem{PhDII}
B.~Chorny.
\newblock Localization with respect to a class of maps. {II}. {E}quivariant
  cellularization and its application.
\newblock {\em Israel J. Math.}, 147:141--155, 2005.

\bibitem{pro-spaces}
B.~Chorny.
\newblock A generalization of {Q}uillen's small object argument.
\newblock {\em Journal of Pure and Applied Algebra}, 204:568--583, 2006.

\bibitem{Chorny-Brownrep}
B.~Chorny.
\newblock Brown representability for space-valued functors.
\newblock {\em Israel J. Math.}, 194(2):767--791, 2013.

\bibitem{Chorny-Rosicky}
B.~Chorny and J.~Rosick{\'y}.
\newblock Class-combinatorial model categories.
\newblock {\em Homology Homotopy Appl.}, 14(1):263--280, 2012.

\bibitem{Farjoun-natural}
E.~Dror~Farjoun.
\newblock Higher homotopies of natural constructions.
\newblock {\em J. Pure Appl. Algebra}, 108(1):23--34, 1996.

\bibitem{Dugger-generation}
D.~Dugger.
\newblock {C}ombinatorial {M}odel {C}ategories have presentation.
\newblock {\em Adv. in Math.}, 164(1):177--201, December 2001.

\bibitem{Goo:calc3}
T.~G. Goodwillie.
\newblock Calculus. {III}. {T}aylor series.
\newblock {\em Geom. Topol.}, 7:645--711 (electronic), 2003.

\bibitem{Hirschhorn}
P.~S. Hirschhorn.
\newblock {\em Model categories and their localizations}, volume~99 of {\em
  Mathematical Surveys and Monographs}.
\newblock American Mathematical Society, Providence, RI, 2003.

\bibitem{Hovey-Shipley-Smith}
M.~Hovey, B.~Shipley, and J.~Smith.
\newblock Symmetric spectra.
\newblock {\em J. Amer. Math. Soc.}, 13(1):149--208, 2000.

\bibitem{Hovey-Strickland}
M.~Hovey and N.~P. Strickland.
\newblock Morava {$K$}-theories and localisation.
\newblock {\em Mem. Amer. Math. Soc.}, 139(666):viii+100, 1999.

\bibitem{Isaksen-strict}
D.~C. Isaksen.
\newblock Strict model structures for pro-categories.
\newblock In {\em Categorical decomposition techniques in algebraic topology
  (Isle of Skye, 2001)}, volume 215 of {\em Progr. Math.}, pages 179--198.
  Birkh\"auser, Basel, 2004.

\bibitem{Jardine-representability}
J.~F. Jardine.
\newblock Representability theorems for presheaves of spectra.
\newblock {\em J. Pure Appl. Algebra}, 215(1):77--88, 2011.

\bibitem{Lydakis}
M.~Lykadis.
\newblock Simplicial functors and stable homotopy theory.
\newblock Preprint, available via Hopf archive, 1998.

\end{thebibliography}

\end{document}